\documentclass[11pt]{article}
\usepackage[francais,english]{babel}
\usepackage[latin1]{inputenc}
\usepackage[top=2cm, bottom=2cm, left=2.5cm, right=2.5cm]{geometry}

\usepackage{graphicx}
\usepackage{pgf,tikz}
\usetikzlibrary{arrows}
\usepackage{amsmath}
\usepackage{enumerate}
\usepackage{float}
\usepackage{amssymb}
\usepackage{amsthm}
\usepackage{mathrsfs}
\usepackage{color}
\usepackage{fancybox}
\usepackage[retainorgcmds]{IEEEtrantools}
\usepackage{mathabx}
\usepackage{cancel}
\usepackage{listings}
\usepackage{ulem}
\usepackage{bbm}
\usepackage{epsfig}
\usepackage[colorlinks=true]{hyperref}
\usepackage{mathdots}
\usepackage{psfrag}

\newtheorem{rem}{Remark}
\newtheorem{thm}{Theorem}
\newtheorem*{thm*}{Theorem}
\newtheorem*{prop*}{Proposition}

\newtheorem{lem}{Lemma}
\newtheorem{cor}{Corollary}

\newtheorem{hypo}{Assumption}
\newcommand{\tn}{\mathrm}
\newcommand{\dis}{\displaystyle}
\newcommand{\R}{\mathbb{R}}

\newcommand{\N}{\mathbb{N}}
\renewcommand{\S}{\mathbb{S}}
\renewcommand{\L}{\mathbb{L}}
\newcommand*{\medcap}{\mathbin{\scalebox{0.75}{\ensuremath{\bigcap}}}}

\newcommand*{\esp}{\vspace{0.5cm}}

\definecolor{xdxdff}{rgb}{0.490196078431,0.490196078431,1.}
\definecolor{wwccff}{rgb}{0.4,0.8,1.}
\definecolor{qqqqff}{rgb}{0.,0.,1.}
\definecolor{cqcqcq}{rgb}{0.752941176471,0.752941176471,0.752941176471}
\definecolor{uuuuuu}{rgb}{0.266666666667,0.266666666667,0.266666666667}

\begin{document}

\title{\bf Local Error Estimates of the Finite Element Method for an Elliptic Problem with a Dirac Source Term}
\author{Silvia BERTOLUZZA, Astrid DECOENE, Lo\"ic LACOUTURE and S\'ebastien MARTIN}
\maketitle

\esp

{\bf Abstract:} The solutions of elliptic problems with a Dirac measure in right-hand side are not $H^1$ and therefore the convergence of the finite element solutions is suboptimal. Graded meshes are standard remedy to recover quasi-optimality, namely optimality up to a log-factor, for low order finite elements in $\L^2$-norm. Optimal (or quasi-optimal for the lowest order case) convergence has been shown in $\L^2$-norm on a subdomain which excludes the singularity. Here, on such subdomains, we show a quasi-optimal convergence for the $H^s$-norm, $s \geqslant 1$, and an optimal convergence in $H^1$-norm for the lowest order case, on a family of quasi-uniform meshes in dimension 2. The study of this problem is motivated by the use of the Dirac measure as a reduced model in physical problems, for which high accuracy of the finite element method at the singularity is not required. Our results are obtained using local Nitsche and Schatz-type error estimates, a weak version of Aubin-Nitsche duality lemma and a discrete inf-sup condition. These theoretical results are confirmed by numerical illustrations.

\esp

{\bf Key  words:} Dirichlet problem, Dirac measure, Green function, finite element method, local error estimates.

\esp

\section{Introduction.}
This paper deals with the accuracy of the finite element method on elliptic problems with a singular right-hand side. More precisely, let us consider the Dirichlet problem
\begin{equation*}
(P_{\delta}) \left\{ \begin{array}{rccl}
-\triangle u_{\delta} & = & \delta_{x_{0}} & \text{in } \Omega,\\
u_{\delta} & =  & 0 & \text{on } \partial \Omega,
\end{array} \right.
\end{equation*}
where $\Omega \subset \R^{2}$ is a bounded open $\mathcal{C}^{\infty}$ domain or a square, and $\delta_{x_{0}}$ denotes the Dirac measure concentrated at a point $x_{0} \in \Omega$ such that $\tn{dist}(x_{0},\partial \Omega) > 0$.

Problems of this type occur in many applications from different areas, like in the mathematical modeling of electromagnetic fields \cite{JJa}. Dirac measures can also be found on the right-hand side of adjoint equations in optimal control of elliptic problems with state constraints \cite{ECa2}. As further examples where such measures play an important role, we mention controllability for elliptic and parabolic equations \cite{CaCl, CaZu, LeMe} and parameter identification problems with pointwise measurements \cite{RaVe}.

Our interest in $(P_{\delta})$ is motivated by the modeling of the movement of a thin structure in a viscous fluid, such as cilia involved in the muco-ciliary transport in the lung \cite{FuBl}. In the asymptotic of a zero diameter cilium with an infinite velocity, the cilium is modelled by a lineic Dirac of force in the source term. In order to make the computations easier, the lineic Dirac of force can be approximated by a sum of punctual Dirac forces distributed along the cilium \cite{LLa}. In this paper, we address a scalar version of this problem: problem $(P_{\delta})$.

In the regular case, namely the Laplace problem with a regular right-hand side, the finite element solution $u^h$ is well-defined and for $u \in H^{k+1}(\Omega)$, we have, for all $0 \leqslant s \leqslant 1$,
\begin{equation}\label{ineclass}
\| u - u^h \|_{s} \leqslant C h^{k+1-s} \| u \|_{k+1},
\end{equation}
where $k$ is the degree of the method \cite{PCi} and $h$ the mesh size. In dimension 1, the solution $u_{\delta}$ of Problem $(P_{\delta})$ belongs to $H^1(\Omega)$, but it is not $H^2(\Omega)$. In this case, the numerical solution $u_{\delta}^h$ and the exact solution $u_{\delta}$ can be computed explicitly. If $x_{0}$ matches with a node of the discretization, $u_{\delta}^h = u_{\delta}$. Otherwise, this equality is true only on the complementary of the element which contains $x_{0}$, and the convergence orders are 1/2 and 3/2 respectively in $H^1$-norm and $\L^2$-norm. In dimension 2, Problem~$(P_{\delta})$ has no $H^1(\Omega)$-solution, and so, although the finite element solution can be defined, the $H^1(\Omega)$-error has no sense and the $\L^2(\Omega)$-error estimates cannot be obtained by the Aubin-Nitsche method without modification.

Let us review the literature about error estimates for problem $(P_{\delta})$, starting with discretizations on quasi-uniform meshes. Babu\~ska \cite{IBa} showed a $\L^2(\Omega)$-convergence of order $h^{1-\varepsilon}$, $\varepsilon > 0$, for a two-dimensional smooth domain. Scott proved in \cite{RSc} an a priori error estimates of order $2-\frac{d}{2}$, where the dimension $d$ is 2 or 3. The same result has been proved by Casas \cite{ECa1} for general Borel measures on the right-hand side.

To the best of our knowledge, in order to improve the convergence order, Eriksson \cite{KEr} was the first who studied the influence of locally refined meshes near $x_{0}$. Using results from \cite{ScWa}, he proved a convergence of order $k$ and $k+1$ in the $W^{1,1}(\Omega)$-norm and the $\L^1(\Omega)$-norm respectively, for approximations with a $P_{k}$-finite element method. Recently, by Apel and co-authors \cite{ApBe}, a $\L^2(\Omega)$-error estimate of order $h^2 | \ln h |^{3/2}$ has been proved in dimension 2, using graded meshes. Optimal convergence rates with graded meshes were also recovered by D'Angelo \cite{CDa} using weighted Sobolev spaces. A posteriori error estimates in weighted spaces have been established by Agnelli and co-authors \cite{AgGa}.

These theoretical a priori results for finite elements using graded meshes increase the complexity of the meshing and the computational cost, even if the mesh is refined only locally, especially if there are several Dirac measures in the right-hand side. Eriksson \cite{KEr2} developed a numerical method to solve the problem and recover the optimal convergence rate: the numerical solution is searched in the form $u_{0} + v_{h}$ where $u_{0}$ contains the singularity of the solution and $v_{h}$ is the numerical solution of a smooth problem. This method has been developped in the case of the Stokes problem in \cite{LLa}.

However, in applications, the Dirac measure at $x_{0}$ is often a model reduction approach, and a high accuracy at $x_{0}$ of the finite element method is not necessary. Thus, it is interesting to study the error on a fixed subdomain which excludes the singularity. Recently, K\"oppl and Wohlmuth have shown in \cite{KoWo} quasi-optimal convergence for low order in $\L^2$-norm for Lagrange finite elements and optimal convergence for higher order. In this paper, we show in dimension 2 a quasi-optimal convergence in $H^s$-norm, $s \geqslant 1$, and an optimal convergence in the case of low order. The $\L^2$-error estimates established in \cite{KoWo} are not used and the proof is based on different arguments. These results imply that graded meshes are not required to recover optimality far from the singularity and that there are no pollution effects.

The paper is organized as follows. Our main results are formulated in section \ref{MainResults} after the Nitsche and Schatz Theorem, which is an important tool for the proof presented in section \ref{ProofTHM}. In section \ref{ParticularCase} another argument is presented to obtain an optimal estimate in the particular case of the $P_{1}$-finite elements. Finally, we illustrate in section \ref{NumIllus} our theoretical results by some numerical simulations.

\esp

\section{Main results.}\label{MainResults}

In this section, we define all the notations used in this paper, formulate our main results and recall an important tool for the proof, the Nitsche and Schatz Theorem.

\subsection{Notations.}\label{notations}

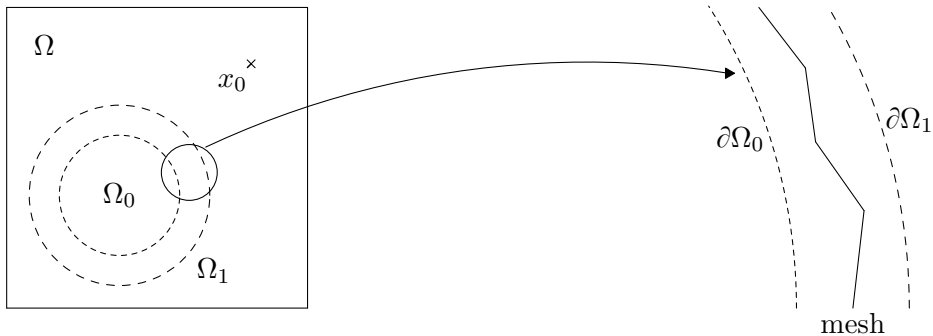
\begin{figure}[H]
\centering

\begin{tikzpicture}[line cap=round,line join=round,>=triangle 45,x=1.0cm,y=1.0cm]
\clip(-5.,-1.) rectangle (9.,5.);
\draw (-4.,4.)-- (0.,4.);
\draw (0.,0.)-- (0.,4.);
\draw (-4.,4.)-- (-4.,0.);
\draw (-4.,0.)-- (0.,0.);
\draw [dash pattern=on 2pt off 2pt] (-2.5,1.5) circle (0.8cm);
\draw [dash pattern=on 3pt off 3pt] (-2.5,1.5) circle (1.2cm);
\draw [shift={(-1.,0.)},dash pattern=on 2pt off 2pt]  plot[domain=0.:0.565701468656,variable=\t]({1.*7.5*cos(\t r)+0.*7.5*sin(\t r)},{0.*7.5*cos(\t r)+1.*7.5*sin(\t r)});
\draw [shift={(0.,0.)},dash pattern=on 3pt off 3pt]  plot[domain=0.:0.525366873815,variable=\t]({1.*8.*cos(\t r)+0.*8.*sin(\t r)},{0.*8.*cos(\t r)+1.*8.*sin(\t r)});
\draw [shift={(3.71469847741,-8.64788700799)}] plot[domain=1.41229472108:2.01,variable=\t]({1.*11.9218984257*cos(\t r)+0.*11.9218984257*sin(\t r)},{0.*11.9218984257*cos(\t r)+1.*11.9218984257*sin(\t r)});
\draw (6.,4.)-- (6.61832985787,3.19474326392);
\draw (6.61832985787,3.19474326392)-- (6.75392766617,2.21165915376);
\draw (6.75392766617,2.21165915376)-- (7.39801725558,1.29637394775);
\draw (7.39801725558,1.29637394775)-- (7.25,0.);
\draw(-1.57218801702,1.80450020058) circle (0.371326175124cm);
\draw (-0.75,3.25)-- ++(-1.5pt,-1.5pt) -- ++(3.0pt,3.0pt) ++(-3.0pt,0) -- ++(3.0pt,-3.0pt);
\draw (-1,3) node {$x_{0}$};
\draw[color=black] (-3.5,3.5) node {$\Omega$};
\draw[color=black] (-2.5,1.5) node {$\Omega_{0}$};
\draw (-1.25,0.5) node {$\Omega_{1}$};
\draw[color=black] (5.75,2.25) node {$\partial \Omega_{0}$};
\draw[color=black] (8,2.5) node {$\partial \Omega_{1}$};
\draw [fill=black,shift={(5.59643629193,3.12456918732)},rotate=270] (0,0) ++(0 pt,2.25pt) -- ++(1.94855715851pt,-3.375pt)--++(-3.89711431703pt,0 pt) -- ++(1.94855715851pt,3.375pt);
\draw (7.25,-0.2) node {mesh};
\end{tikzpicture}

\caption{Domains $\Omega_{0}$ and $\Omega_{1}$.}
\label{hypNS}
\end{figure}

For a domain $D$, we will denote by $\| \cdot \|_{s,p,D}$ (respectively $| \cdot |_{s,p,D}$) the norm (respectively the semi-norm) of the Sobolev space $W^{s,p}(D)$, while $\| \cdot \|_{s,D}$ (respectively $| \cdot |_{s,D}$) will stand for the norm (respectively the semi-norm) of the Sobolev space $H^{s}(D)$.

For the numerical solution, let us introduce a family of quasi-uniform simplicial triangulations $\mathcal{T}_{h}$ of $\Omega$ and an order $k$ finite element space $V_{h}^k \subset H_{0}^1(\Omega)$. To ensure that the numerical solution is well-defined, the space $V_{h}^k$ is assumed to contain only continuous functions. The finite element solution $u_{\delta}^h \in V_{h}^k$ of problem $(P_{\delta})$ is defined by
\begin{equation}\label{Galerkinproj}
\int_{\Omega} \nabla u_{\delta}^h \cdot \nabla v_{h} = v_{h} (x_{0}), \ \forall v_{h} \in V_{h}^k.
\end{equation}
We will also evaluate the $H^s$-norm of the error on a subdomain of $\Omega$ which does not contain the singularity, for $s \geqslant 2$, and, whenever we do so, we will of course assume the finite elements $H^s$-conforming. We fix two subdomains of $\Omega$, named $\Omega_{0}$ and $\Omega_{1}$, such that $\Omega_{0} \subset \subset \Omega_{1} \subset \subset \Omega$ and $x_{0} \notin \overline{\Omega_{1}}$ (see Figure~\ref{hypNS}). We consider a mesh which satisfies the following condition:
\begin{hypo}\label{maillage}
For some $h_{0}$, we have for all $0 < h \leqslant h_{0}$ (see Figure \ref{hypNS}),
\begin{equation*}
\overline{\Omega}_{0}^m \medcap \Omega_{1}^{c} = \emptyset, \text{ where } \overline{\Omega}_{0}^m = \dis{\bigcup_{\substack{T \in \mathcal{T}_{h} \\[0.2mm] T \scalebox{0.625}{\ensuremath{\bigcap}} \Omega_{0} \neq \emptyset}}} T,
\end{equation*}
and $\Omega_{1}^{c}$ is the complement of $\Omega_{1}$ in $\Omega$.
\end{hypo}

\esp

\subsection{Regularity of the solution $u_{\delta}$.}\label{PartReg}
In this subsection, we focus on the singularity of the solution, which is the main difficulty in the study of this kind of problems. In dimension 2, problem $(P_{\delta})$  has a unique variational solution $u_{\delta} \in W^{1,p}_{0}(\Omega)$ for all $p\in [1,2[$ (see for instance \cite{ArBe}). Indeed, denoting by $G$ the Green function, $G$ is defined by
\begin{equation*}
G(x) = - \dis{\frac{1}{2\pi}} \log (| x |).
\end{equation*}
This function $G$ satisfies $-\triangle G = \delta_{0}$, so that $G(\cdot - x_{0})$ contains the singular part of $u_{\delta}$. As it is done in \cite{ArBe}, the solution $u_{\delta}$ can be built by adding to $G(\cdot - x_{0})$ a corrector term $\omega \in H^1(\Omega)$, solution of the Laplace( problem
\begin{equation}\label{poisson}
\left\{ \begin{array}{rccl} -\triangle \omega & = & 0 & \text{in } \Omega, \\ \omega & = & -G(\cdot-x_{0}) & \text{on } \partial \Omega. \end{array} \right.
\end{equation}
Then, the solution is given by
\begin{equation*}
u_{\delta}(x) = G(x-x_{0}) + \omega(x) = -\frac{1}{2\pi}\log(| x-x_{0} |) + \omega(x).
\end{equation*}

It is easy to verify that $u_{\delta} \notin H_{0}^{1}(\Omega)$. Actually, we can specify how the quantity $\| u_{\delta} \|_{1,p,\Omega}$ goes to infinity when $p$ goes to $2$, with $p <2$. According to the foregoing, if we write $u_{\delta} = G + \omega$, since $\omega \in H^1(\Omega)$, estimating $\|u_{\delta} \|_{1,p,\Omega}$ when $p$ converges to $2$ from below (which will be denoted by $p \nearrow 2$) is reduced to estimate $\| G \|_{1,p,B}$, where $B = B(0,1)$: $G \in \L^p(\Omega)$ for all $1 \leqslant p < \infty$, and using polar coordinates, we get, for $p < 2$,
\begin{equation*}
| G |_{1,p,B}^p = \int_{B} | \nabla G (x) |^p {\tn d}x = \int_{0}^1 \int_{0}^{2\pi} \left(\frac{1}{2\pi} \frac{1}{r} \right)^p r{\tn d}\theta {\tn d}r = (2\pi)^{1-p} \int_{0}^1 r^{1-p} {\tn d}r = \frac{(2\pi)^{1-p}}{2-p}.
\end{equation*}
Finally, when $p \nearrow 2 $,
\begin{equation}\label{equivud2}
\| u_{\delta} \|_{1,p,\Omega} \sim \frac{1}{\sqrt{2\pi}} \frac{1}{\sqrt{2-p}}.
\end{equation}

\esp

\subsection{Nitsche and Schatz Theorem.}\label{PartNS}
Before stating the Nitsche and Schatz Theorem, let us introduce some known properties of the finite element spaces $V_{h}^k$.

\begin{hypo}\label{hypoNS}
Given two fixed concentric spheres $B_{0}$ and $B$ with $B_{0} \subset \subset B \subset \subset \Omega$, there exists an $h_{0}$ such that for all $0<h\leqslant h_{0}$, we have for some $R \geqslant 1$ and $M > 1$:
\begin{itemize}
\item[{\bf B1}] For any $0 \leqslant s \leqslant R$ and $s \leqslant \ell \leqslant M$, for each $u \in H^{\ell}(B)$, there exists $\eta \in V_{h}^k$ such that
\begin{equation*}
\| u - \eta \|_{s,B} \leqslant C h^{\ell-s} \| u \|_{\ell,B}.
\end{equation*}
Moreover, if $u \in H^{1}_{0}(B_{0})$ then $\eta$ can be chosen to satisfy $\eta \in H^1_{0}(B)$.
\item[{\bf B2}] Let $\varphi \in \mathscr{C}^{\infty}_{0}(B_{0})$ and $u_{h} \in V_{h}^k$, then there exists $\eta \in V_{h}^k \medcap H_{0}^{1}(B)$ such that
\begin{equation*}
\| \varphi u_{h} - \eta \|_{1,B} \leqslant C(\varphi,B,B_{0}) h \| u_{h} \|_{1,B}.
\end{equation*}
\item[{\bf B3}] For each $h \leqslant h_{0}$ there exists a domain $B_{h}$ with $B_{0} \subset \subset B_{h} \subset \subset B$ such that if $0 \leqslant s \leqslant \ell \leqslant R$ then for all $u_{h} \in V_{h}^k$ we have
\begin{equation*}
\| u_{h} \|_{\ell,B_{h}} \leqslant C h^{s-\ell} \| u_{h} \|_{s,B_{h}}.
\end{equation*}
\end{itemize}
\end{hypo}

We now state the following theorem, a key tool in the forthcoming proof of Theorem \ref{THM}.

\begin{thm*}[Nitsche and Schatz \cite{NiSc}]\label{NS}
Let $\Omega_{0} \subset \subset \Omega_{1} \subset \subset \Omega$ and let $V_{h}^k$ satisfy Assumption \ref{hypoNS}. Let $u \in H^{\ell}(\Omega_{1})$, let $u_{h} \in V_{h}^k$ and let $q$ be a nonnegative integer, arbitrary but fixed. Let us suppose that $u-u_{h}$ satisfies
\begin{equation*}
\int_{\Omega} \nabla (u-u_{h}) \cdot \nabla v_{h} = 0,\ \forall v_{h} \in V_{h}^k \medcap H_{0}^1(\Omega_{1}).
\end{equation*}
Then there exists $h_{1}$ such that if $h \leqslant h_{1}$ we have
\begin{enumerate}[(i)]
\item for $s=0,1$ and $1 \leqslant \ell \leqslant M$,
\begin{equation*}
\| u-u_{h} \|_{s,\Omega_{0}} \leqslant C \left( h^{\ell-s} \| u \|_{\ell,\Omega_{1}} + \|u - u_{h} \|_{-q,\Omega_{1}} \right),
\end{equation*}
\item for $2 \leqslant s \leqslant \ell \leqslant M$ and $s \leqslant k < R$,
\begin{equation*}
\| u-u_{h} \|_{s,\Omega_{0}} \leqslant C \left( h^{\ell-s} \| u \|_{\ell,\Omega_{1}} + h^{1-s} \|u - u_{h} \|_{-q,\Omega_{1}} \right).
\end{equation*}
\end{enumerate}
\end{thm*}

In this paper, we will actually need a more general version of the assumptions on the approximation space $V_{h}^k$:

\begin{hypo}\label{hypoGene}
Given $B \subset \Omega$, let be $p' \geqslant 2$, there exists an $h_{0}$ such that for all $0<h\leqslant h_{0}$, we have for some $R \geqslant 1$ and $M > 1$:
\begin{itemize}
\item[$\widetilde{\rm {\bf B}}${\bf 1}] For any $0 \leqslant s \leqslant R$ and $s \leqslant \ell \leqslant M$, for each $u \in H^{\ell}(B)$, there exists $\eta \in V_{h}^k$ such that, for any finite element $T \subset B$,
\begin{equation*}
| u - \eta |_{s,p',T} \leqslant C h^{d(1/p' - 1/2)} h^{\ell-s} | u |_{\ell,2,T}.
\end{equation*}
\item[$\widetilde{\rm {\bf B}}${\bf 3}] For $0 \leqslant s \leqslant \ell \leqslant R$, for all $u_{h} \in V_{h}^k$, for any finite element $T$ in the family $\mathcal{T}_{h}$, we have
\begin{equation*}
\| u_{h} \|_{\ell,p',T} \leqslant C h^{d(1/p' - 1/2)} h^{s-\ell} \| u_{h} \|_{s,2,T}.
\end{equation*}
\end{itemize}
\end{hypo}

Assumptions $\widetilde{\rm {\bf B}}${\bf 1} and $\widetilde{\rm {\bf B}}${\bf 3} are generalisations of assumptions {\bf B1} and {\bf B3}. They are quite standard and satisfied by a wide variety of approximation spaces, including all finite element spaces defined on quasi-uniform meshes \cite{PCi}. The parameters $R$ and $M$ play respectively the role of the regularity and order of approximation of the approximation space $V_{h}^k$. For example, in the case of $P_{1}$-finite elements, we have $R=3/2-\varepsilon$ and $M=2$. Assumption {\bf B2} is less common but also satisfied by a wide class of approximation spaces. Actually, for Lagrange and Hermite finite elements, a stronger property than assumption {\bf B2} is shown in \cite{SBe}: let $0 \leqslant s \leqslant \ell \leqslant k$, $\varphi \in \mathscr{C}^{\infty}_{0}(B)$ and $u_{h} \in V_{h}^k$, then there exists $\eta \in V_{h}^k$ such that
\begin{equation}\label{comdis}
\| \varphi u_{h} - \eta \|_{s,B} \leqslant C(\varphi) h^{\ell-s+1} \| u_{h} \|_{\ell,B}.
\end{equation}
Applied for $s=\ell=1$, inequality \eqref{comdis} gives assumption {\bf B2}.

\esp

\subsection{Statement of our main results.}
Our main results are Theorems \ref{THM} and Theorem \ref{THM1}. The rest of the paper is mostly concerned by the proof and the illustration of the following theorems.

\begin{thm}\label{THM}
Let $\Omega_{0} \subset \subset \Omega_{1} \subset \subset \Omega$ satisfy Assumption \ref{maillage}, $ 1 \leqslant s \leqslant k$. Let $u_{\delta}$ be the solution of problem $(P_{\delta})$ and $u_{\delta}^h$ its Galerkin projection onto $V_{h}^k$, satisfying \eqref{Galerkinproj}. Under Assumptions \ref{hypoNS} and \ref{hypoGene}, there exists $h_{1}$ such that if $0 < h \leqslant h_{1}$, we have,
\begin{equation}\label{INEQTHM}
\| u_{\delta} - u_{\delta}^{h} \|_{1,\Omega_{0}} \leqslant C(\Omega_{0},\Omega_{1},\Omega) h^{k} \sqrt{| \ln h |}.
\end{equation}
In addition, for $s \geqslant 2$, if the finite elements are supposed $H^k$-conforming, we have
\begin{equation}\label{INEQTHMk}
\| u_{\delta} - u_{\delta}^{h} \|_{s,\Omega_{0}} \leqslant C(\Omega_{0},\Omega_{1},\Omega) h^{k+1-s} \sqrt{| \ln h |}.
\end{equation}
\end{thm}

We can show a stronger result than Theorem \ref{THM} for $P_{1}$-finite elements. With the same notations and under the same assumptions,

\begin{thm}\label{THM1}
The $P_{1}$-finite element method converges with the order 1 for $H^1(\Omega_{0})$-norm. More precisely:
\begin{equation*}
\| u_{\delta} - u_{\delta}^{h} \|_{1,\Omega_{0}} \leqslant C(\Omega_{0},\Omega_{1},\Omega)h.
\end{equation*}
\end{thm}

\esp

\section{Proof of Theorem \ref{THM}.}\label{ProofTHM}
This section is devoted to the proof of Theorem \ref{THM}. We first show a weak version of the Aubin-Nitsche duality lemma (Lemma \ref{ANweak}) and establish a discrete inf-sup condition (Lemma \ref{infsupdiscrete}). Then, we use these results to prove Theorem \ref{THM}

\subsection{Aubin-Nitsche duality lemma with a singular right-hand side.}
The proof of Theorem \ref{THM} is based on Nitsche and Schatz Theorem. In order to estimate the quantity $\| u_{\delta} - u_{\delta}^h \|_{-q,\Omega_{1}}$, we will first show a weak version of Aubin-Nitsche Lemma, in the case of Poisson Problem with a singular right-hand side.

\begin{lem}\label{ANweak}
Let $f \in W^{-1,p}(\Omega) = (W^{1,p'}_{0}(\Omega))'$, $1 < p < 2$, and $u \in W^{1,p}_{0}(\Omega)$ be the unique solution of
\begin{equation*}
\left\{ \begin{array}{rccl}
- \triangle u & = & f &\text{in }\Omega,\\
u & = & 0 &\text{on } \partial \Omega.
\end{array}\right.
\end{equation*}
Let $u_{h} \in V_{h}^k$ be the Galerkin projection of $u$. For finite elements of order $k$, letting $e = u - u_{h}$,  we have for all $0 \leqslant q \leqslant k-1$,
\begin{equation}\label{AubNitWeak}
\| e \|_{-q,\Omega} \leqslant C h^{q+1} h^{2(1/p'-1/2)} | e |_{1,p,\Omega}.
\end{equation}
\end{lem}

\begin{proof}
We aim at estimating, for $q \geqslant 0$, the $H^{-q}$-norm of the error $e$:
\begin{equation}\label{defnormneg}
\| e \|_{-q,\Omega} = \sup_{\phi \in \mathscr{C}^{\infty}_{0}(\Omega)} \frac{| \int_{\Omega} e \phi |}{ \| \phi \|_{q,\Omega}}.
\end{equation}
The error $e \in W^{1,p}_{0}$ satisfies
\begin{equation*}
\int_{\Omega} \nabla e \cdot \nabla v_{h} = 0, \ \forall v_{h} \in V_{h}^k.
\end{equation*}
Let be $\phi \in \mathscr{C}^{\infty}_{0}(\Omega)$ and let $w^\phi \in H^{q+2}$ be the solution of
\begin{equation*}
\left\{ \begin{array}{rccl}
- \triangle w^\phi & = & \phi &\text{in }\Omega,\\
w^\phi & = & 0 &\text{on } \partial \Omega.
\end{array} \right.
\end{equation*}
In dimension 2, by the Sobolev injections established for instance in \cite{HBr}, $H^{q+2}(\Omega) \subset W^{1,p'}(\Omega)$ for all $p'$ in $[2, +\infty [$. Thus, for any $w_{h} \in V_{h}^k$,
\begin{equation*}
\left| \int_{\Omega} e \phi \right| = \left| \int_{\Omega} e \triangle w^\phi \right| = \left| \int_{\Omega} \nabla e \cdot \nabla w^\phi \right| = \left| \int_{\Omega} \nabla e \cdot \nabla(w^\phi - w_{h}) \right| \leqslant | w^\phi - w_{h} |_{1,p',\Omega} | e |_{1,p,\Omega}.
\end{equation*}
We have to estimate $| w^\phi - w_{h} |_{1,p',\Omega}$. It holds
\begin{equation*}
|w^\phi - w_{h}|_{1,p',\Omega}^{p'} = \sum_{T} |w^\phi - w_{h}|_{1,p',T}^{p'}.
\end{equation*}
For all $0 \leqslant q \leqslant k-1$ and for all element T in $\mathcal{T}_{h}$, thanks to Assumption $\widetilde{\rm {\bf B}}${\bf 1} applied for $s=1$, $\ell = q+2$, there exists $w_{h} \in V_{h}^k$ such as
\begin{equation}\label{ineCiarlet}
| w^\phi-w_{h}|_{1,p',T} \leqslant C h^{2(1/p' - 1/2)} h^{q+1} | w^\phi |_{q+2,2,T}.
\end{equation}
We number the triangles of the mesh $\{T_{i}, i=1, \cdots, N \}$ and we set
\begin{equation*}
a = (a_{i})_{i} \text{ and } b=(b_{i})_{i}, \text{ where } a_{i} = | w^\phi - w_{h} |_{1,p',T_{i}} \text{ and } b_{i} = | w^\phi |_{q+2,2,T_{i}}.
\end{equation*}
By \eqref{ineCiarlet}, we have, for all $i$ in $[\![1,N]\!]$,
\begin{equation*}
a_{i} \leqslant C h^{2(1/p'-1/2)} h^{q+1}b_{i}.
\end{equation*}
We recall the norm equivalence in $\R^N$ for $0 < r < s$,
\begin{equation*}
\| x \|_{\ell^s} \leqslant \| x \|_{\ell^r} \leqslant N^{1/r-1/s} \| x \|_{\ell^s}.
\end{equation*}
Remark that here $N \sim Ch^{-2}$. As $2 < p'$, we have $\| b \|_{\ell^{p'}} \leqslant \| b \|_{\ell^{2}}$. Then, we can write
\begin{align*}
| w^\phi - w_{h} \|_{1,p',\Omega} = \| a \|_{\ell^{p'}} & \leqslant h^{q+1} h^{2(1/p'-1/2)} \| b \|_{\ell^{p'}}\\
& \leqslant h^{q+1} h^{2(1/p'-1/2)} \| b \|_{\ell^{2}} \\
& \leqslant h^{q+1} h^{2(1/p'-1/2)} | w^\phi |_{q+2,2,\Omega}\\
& \leqslant h^{q+1} h^{2(1/p'-1/2)} \| \phi \|_{q,\Omega}.
\end{align*}
Finally, using this estimate in \eqref{defnormneg}, we obtain, for $q \leqslant k-1$,
\begin{equation*}
\| e \|_{-q,\Omega} \leqslant C h^{q+1} h^{2(1/p'-1/2)} | e |_{1,p,\Omega}.
\end{equation*}
\end{proof}

\begin{cor}
For finite elements of order $k$, for any $\varepsilon > 0$ small enough,
\begin{equation}\label{ineqinterm1}
\| u_{\delta} - u_{\delta}^h \|_{-k+1,\Omega} \leqslant C h^{k} h^{-\varepsilon} | u_{\delta} - u_{\delta}^h |_{1,p,\Omega}.
\end{equation}
\end{cor}

\begin{proof}
We will apply Lemma \ref{ANweak} to estimate $\| u_{\delta} - u_{\delta}^h \|_{-q,\Omega}$. This can be done with a suitable choice for $(p,p')$. With $\varepsilon > 0$ small enough, let us take:
\begin{equation}\label{defp}
p = \frac{2}{1+\varepsilon} \text{ and } p' = \frac{2}{1-\varepsilon}.
\end{equation}
In inequality \eqref{AubNitWeak}:
\begin{equation}\label{calculexp}
2 \left( \frac{1}{p'} - \frac{1}{2} \right) = 2 \left( \frac{1-\varepsilon}{2} - \frac{1}{2} \right) = - \varepsilon.
\end{equation}
Finally, for finite elements of order $k$,
\begin{equation*}
\| u_{\delta} - u_{\delta}^h \|_{-k+1,\Omega} \leqslant C h^{k} h^{-\varepsilon} | u_{\delta} - u_{\delta}^h |_{1,p,\Omega}.
\end{equation*}
\end{proof}

\esp

\subsection{Estimate of $| u_{\delta}-u_{\delta}^h |_{1,p,\Omega}$.}
It remains to estimate the quantity $| u_{\delta}-u_{\delta}^h |_{1,p,\Omega}$ by bounding $| u_{\delta}^h |_{1,p,\Omega}$ in terms of $|u_{\delta}|_{1,p,\Omega}$ (equality \eqref{ineqinterm2}). To achieve this, we will need the following discrete inf-sup condition.

\begin{lem}\label{infsupdiscrete}
For $p$ and $p'$ defined in \eqref{defp}, we have the discrete inf-sup condition
\begin{equation*}
\inf_{u_{h} \in V_{h}^k} \sup_{v_{h} \in V_{h}^k} \frac{\int_{\Omega}\nabla u_{h} \cdot \nabla v_{h}}{\| u_{h} \|_{1,p,\Omega} \| v_{h} \|_{1,p',\Omega}} \geqslant C h^{\varepsilon}.
\end{equation*}
\end{lem}

\begin{proof}
For $\varepsilon$ sufficiently small and $p$ and $p'$ defined in \eqref{defp}, the continuous inf-sup condition
\begin{equation*}
\inf_{u \in W^{1,p}_{0}} \sup_{v \in W^{1,p'}_{0}} \frac{\int_{\Omega} \nabla u \cdot \nabla v}{\| u \|_{1,p} \| v \|_{1,p'}} \geqslant \beta > 0
\end{equation*}
holds for $\beta$ independant of $p$ and $p'$. It is a consequence of the duality of the two spaces $W^{1,p}_{0}(\Omega)$ and $W^{1,p'}_{0}(\Omega)$, see \cite{JeKe}. For $v \in W^{1,p'}_{0}(\Omega)$, let $\Pi_{h}v$ denote the $H^1_{0}$-Galerkin projection of $v$ onto $V_{h}^k$. This is well defined since $W^{1,p'}(\Omega) \subset H^1_{0}(\Omega)$. We apply Assumption $\widetilde{\rm {\bf B}}${\bf 3} to $\Pi_{h}v$ for $\ell = s = 1$, and get
\begin{equation*}
\| \Pi_{h}v \|_{1, p',\Omega} \leqslant C h^{-2(1/2 - 1/p')} \| \Pi_{h}v \|_{1,2,\Omega} \leqslant C h^{-2(1/2 - 1/p')} \| v \|_{1,2,\Omega} \leqslant C h^{-2(1/2 - 1/p')} \| v \|_{1,p',\Omega}.
\end{equation*}
Moreover, for any $u_{h} \in V_{h}^k \subset W^{1,p}(\Omega)$,
\begin{align*}
\| u_{h} \|_{1,p,\Omega} & \leqslant C \sup_{v \in W^{1,p'}_{0}} \frac{\int_{\Omega} \nabla u_{h} \cdot \nabla v}{\| v \|_{1,p',\Omega}} = C \sup_{v \in W^{1,p'}_{0}} \frac{\int_{\Omega} \nabla u_{h} \cdot \nabla \Pi_{h}v}{\| v \|_{1,p',\Omega}} \\
& \leqslant C h^{-2(1/2 - 1/p')} \sup_{v \in W^{1,p'}_{0}} \frac{\int_{\Omega} \nabla u_{h} \cdot \nabla \Pi_{h}v}{\| \Pi_{h}v \|_{1,p',\Omega}} \\
& \leqslant C h^{-2(1/2 - 1/p')} \sup_{v_{h} \in V_{h}^k} \frac{\int_{\Omega} \nabla u_{h} \cdot \nabla v_{h}}{\| v_{h} \|_{1,p',\Omega}}.
\end{align*}
Finally, thanks to Poincar\'e inequality, and to inequality \eqref{calculexp},
\begin{equation*}
\inf_{u_{h} \in V_{h}^k} \sup_{v_{h} \in V_{h}^k} \frac{\int_{\Omega}\nabla u_{h} \cdot \nabla v_{h}}{\| u_{h} \|_{1,p,\Omega} \| v_{h} \|_{1,p',\Omega}} \geqslant C h^{\varepsilon}.
\end{equation*}
\end{proof}

Then, we can estimate $| u_{\delta} - u_{\delta}^h |_{1,p,\Omega}$ :
\begin{lem}
With $p$ and $p'$ defined in \eqref{defp},
\begin{equation}\label{ineqinterm3}
| u_{\delta} - u_{\delta}^h |_{1,p,\Omega} \leqslant C\frac{h^{-\varepsilon}}{\sqrt{\varepsilon}}.
\end{equation}
\end{lem}

\begin{proof}
According to Lemma \ref{infsupdiscrete}, it exists $v_{h} \in V_{h}^k$, with $\| v_{h} \|_{1,p',\Omega} = 1$, such that
\begin{equation*}
h^{2(1/2 - 1/p')} \| u_{\delta}^h \|_{1,p,\Omega} \leqslant C \int_{\Omega} \nabla u_{\delta}^h \cdot \nabla v_{h} = C \int_{\Omega} \nabla u_{\delta} \cdot \nabla v_{h} \leqslant C \| u_{\delta} \|_{1,p,\Omega}.
\end{equation*}
So we have
\begin{equation}\label{ineqinterm2}
| u_{\delta} - u_{\delta}^h |_{1,p,\Omega} \leqslant | u_{\delta} |_{1,p,\Omega} + | u_{\delta}^h |_{1,p,\Omega} \leqslant C h^{-2(1/2 - 1/p')} \| u_{\delta} \|_{1,p,\Omega}.
\end{equation}
All that remains is to substitute $\| u_{\delta}\|_{1,p,\Omega}$ for the expression established in \eqref{equivud2}. For $p$ defined as in \eqref{defp},
\begin{equation*}
\| u_{\delta} \|_{1,p,\Omega} \leqslant \frac{C}{\sqrt{2-p}} \leqslant \frac{C}{\sqrt{\varepsilon}}.
\end{equation*}
Finally, with \eqref{calculexp} and \eqref{ineqinterm2}, we get
\begin{equation*}
| u_{\delta} - u_{\delta}^h |_{1,p,\Omega} \leqslant C\frac{h^{-\varepsilon}}{\sqrt{\varepsilon}}.
\end{equation*}
\end{proof}

\esp

\subsection{Proof of Theorem \ref{THM}.}
We can now prove Theorem \ref{THM}.

\begin{proof}
The function $u_{\delta}$ is analytic on $\overline{\Omega}_{1}$, therefore the quantity $\| u_{\delta} \|_{k+1,\Omega_{1}}$ is bounded. If we suppose $s = 1$, Nitsche and Schatz Theorem gives, for $\ell = k+1$ and $q = k-1$,
\begin{equation*}
\| u_{\delta}-u_{\delta}^h \|_{1,\Omega_{0}} \leqslant C \left( h^{k} + \|u_{\delta} - u_{\delta}^h \|_{-k+1,\Omega_{1}} \right).
\end{equation*}
Thanks to \eqref{ineqinterm1} and \eqref{ineqinterm3},
\begin{equation*}
\| u_{\delta} - u_{\delta}^h \|_{-k+1,\Omega} \leqslant C h^k \frac{h^{-2\varepsilon}}{\sqrt{\varepsilon}},
\end{equation*}
therefore, taking $\varepsilon = | \ln h |^{-1}$,
\begin{equation}\label{ineqinterm5-2}
\| u_{\delta} - u_{\delta}^h \|_{-k+1,\Omega} \leqslant C h^k \sqrt{| \ln h |}.
\end{equation}
Finally, we get the result of Theorem \ref{THM} for $s=1$ (inequality \eqref{INEQTHM}):
\begin{equation*}
\| u_{\delta} - u_{\delta}^h \|_{1,\Omega_{0}} \leqslant C h^k \sqrt{| \ln h |}.
\end{equation*}

Now, let us fix $2 \leqslant s \leqslant k$, Nitsche and Schatz Theorem gives, for $\ell = k+1$ and $q = k-1$,
\begin{equation*}
\| u_{\delta}-u_{\delta}^h \|_{s,\Omega_{0}} \leqslant C \left( h^{k+1-s} + h^{1-s} \| u_{\delta} - u_{\delta}^h \|_{-k+1,\Omega_{1}} \right).
\end{equation*}
So, thanks to \eqref{ineqinterm5-2}, we get the second result of Theorem \ref{THM} (inequality \eqref{INEQTHMk}),
\begin{equation*}
\| u_{\delta} - u_{\delta}^h \|_{s,\Omega_{0}} \leqslant C h^{k+1-s} \sqrt{| \ln h |}.
\end{equation*}
which ends the proof of Theorem \ref{THM}.
\end{proof}

\esp

\section{Proof of Theorem \ref{THM1}.}\label{ParticularCase}
To prove this theorem, we first regularize the right-hand side and prove that in our case the solution $u_{\delta}$ of $(P_{\delta})$ and the solution of the regularized problem are the same on the complementary of a neighbourhood of the singularity (Theorem \ref{bal}). The proof of Theorem \ref{THM1} is based once again on the Nitsche and Schatz Theorem and the observation that the discrete right-hand sides of problem $(P_{\delta})$ and the regularized problem are exactely the same, so that the numerical solutions are the same too (Lemma \ref{sameRHS}).

\subsection{Direct problem and regularized problem.}
Let $\varepsilon >0$, and $f_{\varepsilon}$ be defined on $\Omega$ by
\begin{equation}\label{deffe}
f_{\varepsilon} = \dis{\frac{d}{\sigma(\mathbb{S}_{d-1}) \varepsilon^{d}}\mathbbm{1}_{B_{\varepsilon}}},
\end{equation}
where $B_{\varepsilon} = B(x_{0},\varepsilon)$ and $\sigma(\mathbb{S}_{d-1})$ is the Lebesgue measure of the unit sphere in dimension $d$. The parameter $\varepsilon$ is supposed to be small enough so that $\overline{B_{\varepsilon}} \subset \subset \Omega$. The function $f_{\varepsilon}$ is a regularisation of de the Dirac distribution $\delta_{x_{0}}$. Let us consider the following problem:
\begin{equation*}
(P_{\varepsilon}) \left\{ \begin{array}{rccl} -\triangle u_{\varepsilon} & = & f_{\varepsilon} & \text{in } \Omega, \\ u_{\varepsilon} & = & 0 & \text{on } \partial \Omega. \end{array} \right.
\end{equation*}
Since $f_{\varepsilon} \in \L^2(\Omega)$, it is possible to show that problem $(P_{\varepsilon})$ has a unique variational solution $u_{\varepsilon}$ in $H^1_{0}(\Omega) \medcap H^2(\Omega)$ \cite{PGr}. We will show the following result:
\begin{thm}\label{bal}
The solution $u_{\delta}$ of $(P_{\delta})$  and the solution $u_{\varepsilon}$ of $(P_{\varepsilon})$  coincide on the closed $\widetilde{\Omega} = \overline{\Omega} \setminus B_{\varepsilon}$, ie,
\begin{equation*}
{u_{\delta}}_{\dis{|_{\widetilde{\Omega}}}} = {u_{\varepsilon}}_{\dis{|_{\widetilde{\Omega}}}}.
\end{equation*}
\end{thm}
The proof is based on the following lemma.

\begin{lem}\label{ConvHarm}
Let $d \in \N^{*}$, $\varepsilon >0$, $x \in \R^d$, $v$ a function defined on $\R^d$, harmonic on $\overline{B}(x,\varepsilon)$, and $f \in L^1(\R^d)$ such that
\begin{itemize}
\item $f$ is radial and positive,
\item $supp(f) \subset B(0,\varepsilon)$, $\varepsilon > 0$,
\item $\dis{\int_{\R^d}} f(x) \, {\tn dx} = 1$.
\end{itemize}
Then, $f * v(x) = \dis{\int_{\R^d}} f(y)v(x-y) \, \tn{d}y = v(x)$.
\end{lem}
\begin{proof}
As $supp(f)\subset B(0,\varepsilon)$, using spherical coordinates, we have:
\begin{equation*}
f * v (x) = \int_{0}^{\varepsilon} \int_{\S^{d-1}} f(r)v(x-r\omega) r^{d-1} \, {\tn d}\omega \, {\tn dr} = \int_{0}^{\varepsilon} r^{d-1} f(r) \left( \int_{\S^{d-1}} v(x-r\omega) {\tn d}\omega \right) {\tn dr}.
\end{equation*}
Besides, $v$ is harmonic on $\overline{B}(x,\varepsilon)$, so that the mean value property gives, for $0 < r \leqslant \varepsilon$,
\begin{equation*}
v(x) = \frac{1}{\sigma(\partial B(x,r))} \int_{\partial B(x,r)} v(y) \, {\tn d}y = \frac{r^{d-1}}{\sigma(\partial B(x,r))} \int_{\S^{d-1}} v(x-r\omega) \, {\tn d}\omega,
\end{equation*}
thus
\begin{equation*}
f * v (x) = \int_{0}^{\varepsilon} f(r) v(x) \sigma (\partial B (x,r) ) \, \tn{d}r = v(x) \int_{0}^{\varepsilon} \int_{\S^{d-1}} f(r) r^{d-1} \, {\tn d}\omega \, {\tn d}r = v(x) \int_{B(0,\varepsilon)} f(y) \, {\tn d}y = v(x).
\end{equation*}
\end{proof}

Now, let us prove Theorem \ref{bal}.
\begin{proof}
First, let us leave out boundary conditions and consider the following problem
\begin{equation}\label{distprob}
- \triangle u = f_{\varepsilon} \text{ in } \mathscr{D}'(\R^d).
\end{equation}
As $-\triangle G = \delta_{0}$ in $\mathscr{D}'(\R^d)$, we can build a function $u$ satisfying \eqref{distprob} as:
\begin{equation*}
u(x)=f_{\varepsilon} * G (x) = \int_{\R^d} f_{\varepsilon}(y) G(x-y) \, {\tn d}y = \int_{\R^d} f_{\varepsilon}(x_{0}+y)G(x-x_{0}-y) \, \tn{d}y = \Big(f_{\varepsilon}(x_{0} + \cdot) * G \Big) (x-x_{0}).
\end{equation*}
Moreover, for all $x\in \Omega \setminus \overline{B_{\varepsilon}}$, $G$ is harmonic on $\overline{B}(x-x_{0},\varepsilon)$, and $f_{\varepsilon}(\cdot + x_{0})$ satisfies the assumptions of Lemma \ref{ConvHarm}, so that $u(x) = \Big( f_{\varepsilon}(x_{0} + \cdot) * G \Big) (x-x_{0}) = G(x-x_{0})$. We conclude that $u$ and $G(\cdot - x_{0})$ have the same trace on $\partial \Omega$, and so $u+\omega$, where $\omega$ is the solution of the Poisson problem \eqref{poisson}, is a solution of the problem $(P_{\varepsilon})$. By the uniqueness of the solution, we have $u_{\varepsilon} = u + \omega$. Finally, for all $x \in \Omega \setminus \overline{B_{\varepsilon}}$, $u_{\varepsilon}(x) = u_{\delta}(x)$. These functions are continuous on $\widetilde{\Omega} = \Omega \setminus B_{\varepsilon}$, this equality is true on the closure of $\Omega$, which ends the proof of Theorem \ref{bal}.
\end{proof}

\begin{rem}
Theorem \ref{bal} holds for any radial positive function $f \in \L^1(\R^d) \medcap \L^2(\R^d)$ such that
\begin{equation*}
supp(f) \subset B(0,\varepsilon) \text{ and } \int_{\R^d} f(x) \, {\tn dx} = 1,
\end{equation*}
taking $f_{\varepsilon} = f(\cdot - x_{0})$. It is a direct consequence of Lemma \ref{ConvHarm}.
\end{rem}

\begin{rem}
The result is true in dimension 1, taking $\dis{f_{\varepsilon} = \frac{1}{2\varepsilon}\mathbbm{1}_{I_{\varepsilon}}}$, where $I_{\varepsilon} = [x_{0}-\varepsilon,x_{0}+\varepsilon] \subset ]a,b[ = I$. In this case, we can easily write down the solutions $u_{\delta}$ and $u_{\varepsilon}$ explicitly,
\begin{align*}
u_{\delta}(x) & = \left\{ 
\begin{array}{ll}
\phantom{-}\dis{\frac{b-x_{0}}{b-a} x - a \frac{b-x_{0}}{b-a}} \phantom{totototototototo} & \text{if } x \in [a,x_{0}],\\[5mm]
\dis{- \frac{x_{0}-a}{b-a} x + b \frac{x_{0}-a}{b-a}} & \text{if } x \in [x_{0},b].
\end{array}%
 \right.
\end{align*}
and observe, as shown in Figure \ref{bal1d}, that $u_{\delta}$ and $u_{\varepsilon}$ coincide outside $I_{\varepsilon}$.
\begin{figure}[H]
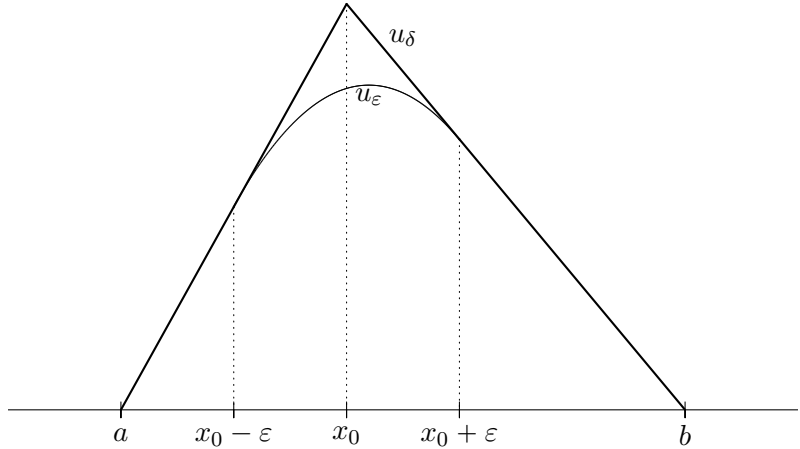

\centering

%

\caption{Illusration of Theorem \ref{bal} in 1D.}
\label{bal1d}
\end{figure}
\end{rem}

\esp

\subsection{Discretizations of the right-hand sides.}
At this point, we introduce a technical assumption on $B_{\varepsilon}$ and the mesh.

\begin{hypo}\label{hypBeTo}
The domain of definition $B_{\varepsilon}$ of the function $f_{\varepsilon}$ is supposed to satisfy
$$B_{\varepsilon} \subset T_{0},$$ where $T_{0}$ denotes the triangle of the mesh which contains the point $x_{0}$ (Figure \ref{cond1}).
\end{hypo}

\begin{figure}[h]
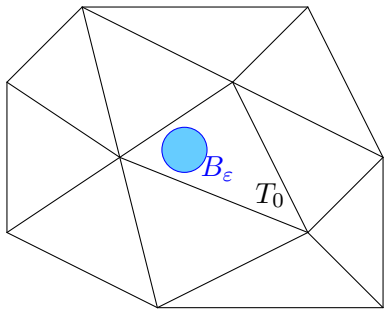

\centering
%
\caption{Assumption on $B_{\varepsilon}$.}
\label{cond1}
\end{figure}

\begin{rem}
The parameter $\varepsilon$ will be chosen to be $h/10$, so it remains to fix a ``good'' triangle $T_{0}$ and to build the mesh accordingly, so that Assumption \ref{hypBeTo} is satisfied. Remark that it is always possible to locally modify any given mesh so that it satisfies this assumption.
\end{rem}

\begin{lem}\label{sameRHS}
Under Assumption \ref{hypBeTo},
$$u_{\varepsilon}^h = u_{\delta}^h,$$
where $u_{\delta}^h$ and $u_{\varepsilon}^h$ are respectively the numerical solutions of problems $(P_{\delta})$ and $(P_{\varepsilon})$.
\end{lem}

\begin{proof}
Let us write down explicitly the discretized right-hand side $F_{\varepsilon}^{h}$ associated to the function $f_{\varepsilon}$: for all node $i$ and associated test function $v_{i} \in V_{h}^k$,
\begin{equation*}
\big( F_{\varepsilon}^{h} \big)_{i} = \int_{\Omega} \frac{1}{\sigma(B_{\varepsilon})} \mathbbm{1}_{B_{\varepsilon}}(x) v_{i}(x) \, {\tn dx} = \int_{B_{\varepsilon} \subset T_{0}} \frac{1}{\sigma(B_{\varepsilon})} v_{i}(x) \, {\tn dx},
\end{equation*}
and $v_{i}$ is affine (and so harmonic) on $T_{0}$, therefore
\begin{equation*}
\big( F_{\varepsilon}^{h} \big)_{i} = \left\{ \begin{array}{cl}
v_{i}(x_{0}) & \text{if } i \text{ is a node of the triangle } T_{0},\\
0 & otherwise.
\end{array} \right.
\end{equation*}
We note that $F_{\varepsilon}^{h} = D^{h}$, where $D^h$ is the discretized right-hand side vector associated to the Dirac mass. That is why, with $A_{h}$ the Laplacian matrix,
\begin{equation*}
u_{\varepsilon}^h - u_{\delta}^h = \sum_{i \text{ node}} \left[ A_{h}^{-1} \big(F_{\varepsilon}^h- D^h \big) \right]_{i} v_{i} = 0.
\end{equation*}
\end{proof}

\begin{rem}
$F_{\varepsilon}^h = D^h$ holds as long as $B_{\varepsilon} \subset T_{0}$. Otherwise, we still have ${u_{\delta}}_{\dis{|_{\Omega_{0}}}} = {u_{\varepsilon}}_{\dis{|_{\Omega_{0}}}}$ (Theorem \ref{bal}), but $F_{\varepsilon}^h \neq D^h$, and so ${u_{\delta}^h}_{\dis{|_{\Omega}}} \neq {u_{\varepsilon}^h}_{\dis{|_{\Omega}}}$.
\end{rem}

\subsection{Proof of Theorem \ref{THM1}.}
Theorem \ref{THM1} can be now proved.

\begin{proof}
First, by triangular inequality, we can write, for $s\in \{ 0,1 \}$:
\begin{equation*}
\| u_{\delta} - u_{\delta}^{h} \|_{s,\Omega_{0}} \leqslant \| u_{\delta} - u_{\varepsilon} \|_{s,\Omega_{0}} + \| u_{\varepsilon} - u_{\varepsilon}^{h} \|_{s,\Omega_{0}} + \| u_{\varepsilon}^{h} - u_{\delta}^{h} \|_{s,\Omega_{0}}.
\end{equation*}
Besides, thanks to Theorem \ref{bal}, we have
\begin{equation}\label{appbal}
\| u_{\delta} - u_{\varepsilon} \|_{s,\Omega_{0}} =0,
\end{equation}
and thanks to Lemma \ref{sameRHS}, we have
\begin{equation*}
\| u_{\delta}^h - u_{\varepsilon}^h \|_{s,\Omega_{0}} =0.
\end{equation*}
Finally we get
\begin{equation}\label{inepre}
\| u_{\delta} - u_{\delta}^{h} \|_{s,\Omega_{0}} \leqslant \| u_{\varepsilon} - u_{\varepsilon}^{h} \|_{s,\Omega_{0}}.
\end{equation}
We will apply Nitsche and Schatz Theorem to $e = u_{\varepsilon} - u_{\varepsilon}^h$. With $\ell=2$, $s=1$, and $p=0$,
\begin{equation}\label{NSapp}
\| e \|_{1,\Omega_{0}} \leqslant C \left( h \|u_{\varepsilon} \|_{2,\Omega_{1}} + \| e \|_{0,\Omega_{1}} \right).
\end{equation}
The domain $\Omega$ is a smooth and $f_{\varepsilon} \in \L^2(\Omega)$, so $u_{\varepsilon} \in H^2(\Omega) \medcap H_{0}^1(\Omega)$, and then, thanks to inequality \eqref{ineclass},
\begin{equation*}
\| e \|_{0,\Omega_{1}} \leqslant \|e \|_{0,\Omega} \leqslant C h^2 \|u_{\varepsilon} \|_{2,\Omega} \leqslant C h^2 \| f_{\varepsilon} \|_{0,\Omega}.
\end{equation*}
As $\| f_{\varepsilon} \|_{0,\Omega}$ can be calculated,
\begin{equation*}
\| f_{\varepsilon} \|_{0,\Omega} = \dis{\left( \int_{\Omega} \left( \frac{1}{\pi \varepsilon^2} \mathbbm{1}_{B_{\varepsilon}}(y) \right)^2 {\tn dy} \right)^{1/2} = \frac{1}{\varepsilon \sqrt{\pi}}},
\end{equation*}
for $\varepsilon \sim h/10$ (in order to respect the assumption on $B_{\varepsilon}$), we get
\begin{equation}\label{errom1}
\|e\|_{0,\Omega_{1}} \leqslant Ch.
\end{equation}
Finally, according to Theorem \ref{bal}, ${u_{\delta}}_{\dis{|_{\Omega_{1}}}} = {u_{\varepsilon}}_{\dis{|_{\Omega_{1}}}}$, therefore combining \eqref{NSapp} and \eqref{errom1}, we get
\begin{equation}\label{erreps1}
\| u_{\varepsilon} - u_{\varepsilon}^h \|_{1,\Omega_{0}} = \| e \|_{1,\Omega_{0}} \leqslant Ch.
\end{equation}
At last, we obtain from inequalities \eqref{inepre} and \eqref{erreps1} the expected error estimate, that is
\begin{equation*}
\| u_{\delta} - u_{\delta}^{h} \|_{1,\Omega_{0}} \leqslant Ch.
\end{equation*}
\end{proof}

\esp

\begin{rem}
In the particular case of the Lagrange finite elements, K\"oppl and Wohlmuth \cite{KoWo} have shown for the lowest order case quasi-optimality and for higher order optimal a priori estimates, in the $\L^2$-norm of a subdomain which does not contain $x_{0}$. The proof is based on Wahlbin-type arguments, which are similar to Nitsche and Schatz Theorem (see \cite{LWa1}, \cite{LWa2}), and different arguments from the ones presented in this paper, like the use of an operator of Scott and Zhang type (see \cite{ScZh}). More precisely, if we introduce a domain $\Omega_{2}$ such that $\Omega_{0} \subset \subset \Omega_{1} \subset \subset \Omega_{2} \subset \subset \Omega$, $x_{0} \notin \overline{\Omega_{2}}$, and satisfying Assumption \ref{maillage}, with our notations and under the assumptions of Theorem \ref{THM}, their main result is
\begin{equation}\label{ineKoWo}
\| u_{\delta} - u_{\delta}^h \|_{0,\Omega_{1}} \leqslant C(\Omega_{1}, \Omega_{2}, \Omega) \left\{ \begin{array}{ll}
h^2 |\ln(h)| & \text{if } k = 1,\\
h^{k+1} & \text{if } k > 1.
\end{array}\right.
\end{equation}
Using this result, we can easily prove the inequality \eqref{INEQTHM}: we apply the Nitsche and Schatz Theorem on $\Omega_{0}$ and $\Omega_{1}$ for $l = k+1$ and $q=0$,
\begin{equation*}
\| u_{\delta}-u_{\delta}^h \|_{1,\Omega_{0}} \leqslant C \left( h^{k} + \|u_{\delta} - u_{\delta}^h \|_{0,\Omega_{1}} \right).
\end{equation*}
Combining it with \eqref{ineKoWo}, we finally get
\begin{equation*}
\| u_{\delta}-u_{\delta}^h \|_{1,\Omega_{0}} \leqslant C \left( h^k + h^{k+1}| \ln(h)| \right) \leqslant C h^k.
\end{equation*}
This result is slightly stronger than inequality \eqref{INEQTHM}, but it is limited to the Lagrange finite elements and the $H^{1}$-norm.
\end{rem}

\esp

\section{Numerical illustrations.}\label{NumIllus}
In this section, we illustrate our theorical results by numerical examples.

\paragraph{Concentration of the error around the singularity.}
First, we present one of the computations which drew our attention to the fact that the convergence could be better far from the singularity. For this example, we define $\Omega$ as the unit disk,
\begin{equation*}
\Omega = \{ x = (x_{1},x_{2}) \in \R^2 : \|x\|_{2} < 1 \},
\end{equation*}
$\Omega_{0}$ as the portion of $\Omega$
\begin{equation*}
\Omega_{0} = \{ x = (x_{1},x_{2}) \in \R^2 :  0.2 <\|x\|_{2} < 1 \},
\end{equation*}
and finally $x_{0} = (0, 0)$ the origin. In this case, the exact solution $u_{\delta}$ of problem $(P_{\delta})$  is given by
\begin{equation*}
u(x) = -\frac{1}{4\pi} \log \Big( x_{1}^2 + x_{2}^2 \Big).
\end{equation*}

When problem $(P_{\delta})$ is solved by the $P_{1}$-finite element method, the numerical solution $u_{\delta}^h$ converges to the exact solution $u_{\delta}$ at the order 1 on the entire domain $\Omega$ for the $\L^2$-norm (see \cite{RSc}). The previous example has shown that the convergence far from the singularity is faster, since the order of convergence in this case is 2 (see \cite{KoWo}). The difference of convergence rates for $\L^2$-norm on $\Omega$ and $\Omega_{0}$ let us suppose that the preponderant part of the error is concentrated around the singularity, as can be seen in Figures \ref{err10}, \ref{err15}, \ref{err20}, and \ref{err30}. Indeed, they respectively show the repartition of the error for $1/h \simeq 10, 15, 20$ and $30$.

\begin{figure}[H]
\centering
\psfrag{Error}{Error}
\psfrag{0.16}{\hspace*{-0.4cm}0.16}
\psfrag{0.12}{0.12}
\psfrag{0.08}{0.08}
\psfrag{0.04}{0.04}
\psfrag{0}{\hspace*{-0.1cm}0}
\begin{minipage}[b]{.45\linewidth}
\centering
\epsfig{figure=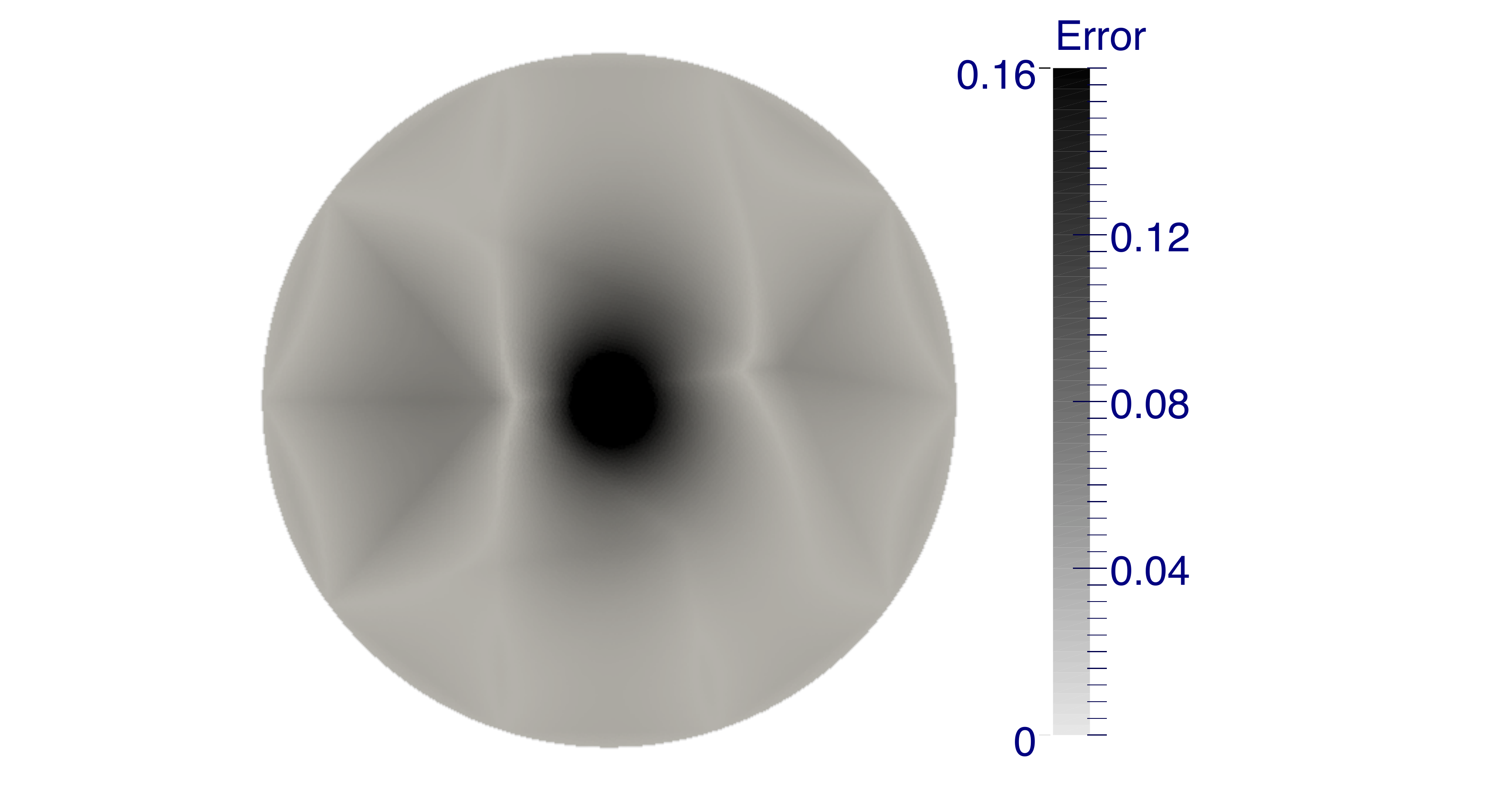,width=\linewidth}
\caption{Error for $1/h \simeq 10$.}
\label{err10}
\end{minipage}
\begin{minipage}[b]{.45\linewidth}
\centering
\epsfig{figure=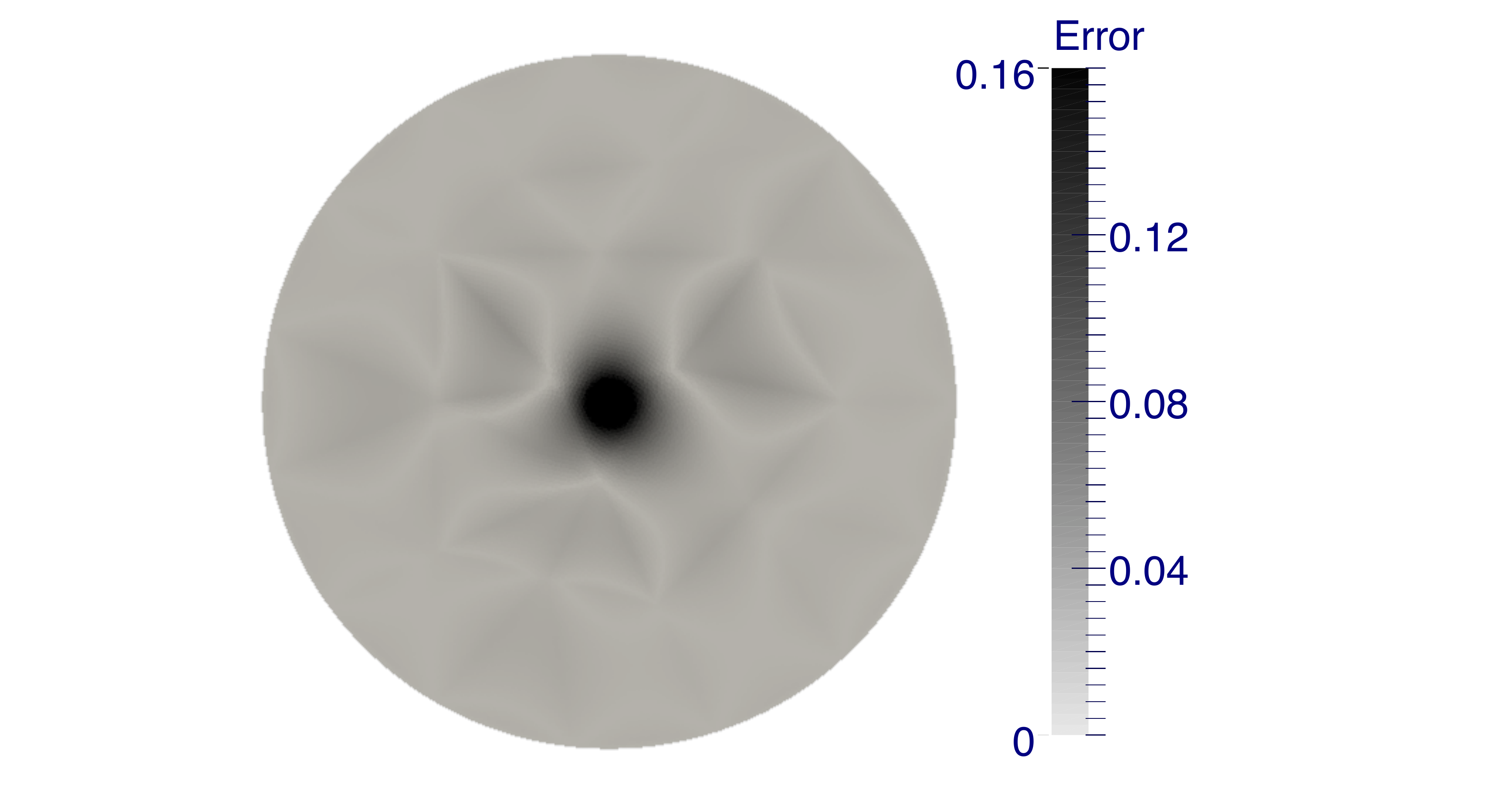,width=\linewidth}
\caption{Error for $1/h \simeq 15$.}
\label{err15}
\end{minipage}
\end{figure}

\begin{figure}[H]
\centering
\psfrag{Error}{Error}
\psfrag{0.16}{\hspace*{-0.4cm}0.16}
\psfrag{0.12}{0.12}
\psfrag{0.08}{0.08}
\psfrag{0.04}{0.04}
\psfrag{0}{\hspace*{-0.1cm}0}
\begin{minipage}[b]{.45\linewidth}
\centering
\epsfig{figure=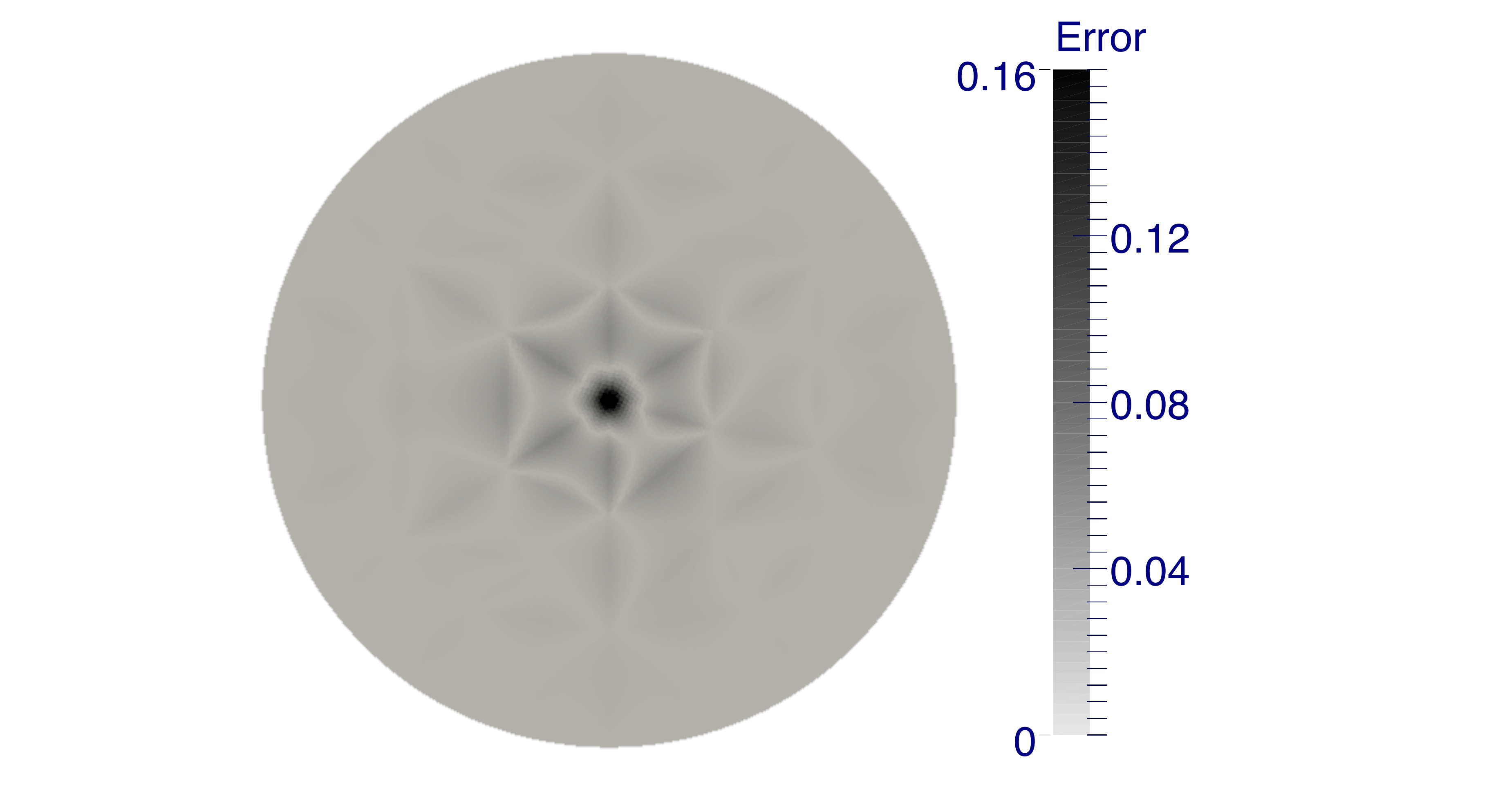,width=\linewidth}
\caption{Error for $1/h \simeq 20$.}
\label{err20}
\end{minipage}
\begin{minipage}[b]{.45\linewidth}
\centering
\epsfig{figure=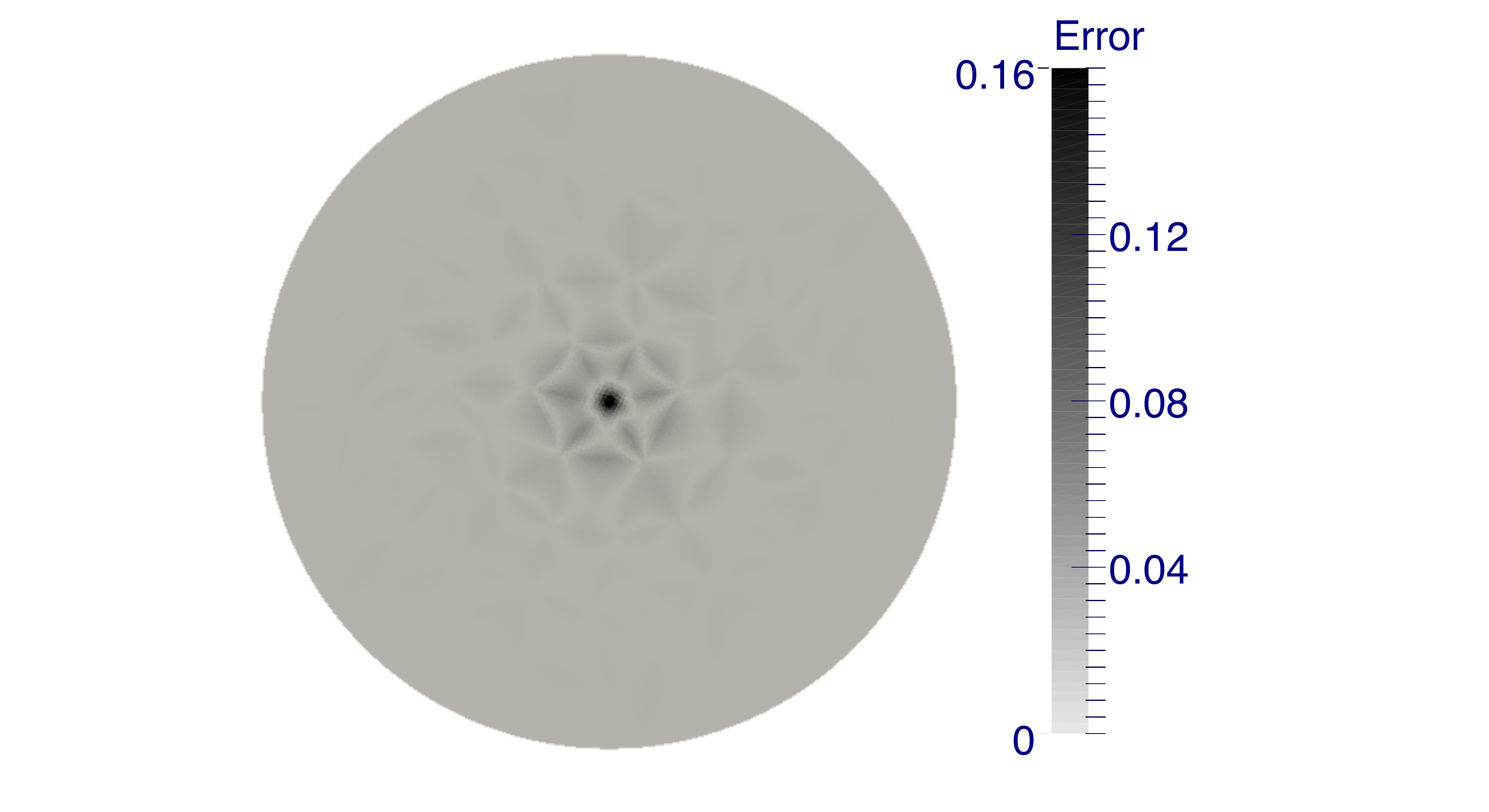,width=\linewidth}
\caption{Error for $1/h \simeq 30$.}
\label{err30}
\end{minipage}
\end{figure}

\esp

\paragraph{Estimated orders of convergence.}
Figure \ref{ordersh1} shows the estimated order of convergence for the $H^1(\Omega_{0})$-norm for the $P_{k}$-finite element method, where $k=1, 2, 3$ and $4$, in dimension 2. The convergence far from the singularity (i.e. excluding a neighborhood of the point $x_{0}$) is the same as in the regular case: the $P_{k}$-finite element method converges at the order $k$ on $\Omega_{0}$ for the $H^1$-norm, as proved in this paper with a $\sqrt {|\ln(h)|} $ multiplier.


\begin{figure}[H]
\centering
\psfrag{Elements P1}{Elements $P_{1}$}
\psfrag{Elements P2}{Elements $P_{2}$}
\psfrag{Elements P3}{Elements $P_{3}$}
\psfrag{Elements P4}{Elements $P_{4}$}
\psfrag{Order=1.0011}{Order = 1.00}
\psfrag{Order=2.0029}{Order = 2.00}
\psfrag{Order=3.0322}{Order = 3.03}
\psfrag{Order=4.2288}{Order = 4.22}
\psfrag{-0.0001}{ \ \ \ \ 1}
\psfrag{-6.9078}{ \ $10^{-3}$}
\psfrag{-13.816}{ \  $10^{-6}$}
\psfrag{-20.723}{ \ $10^{-9}$}
\psfrag{-27.631}{ \ $10^{-12}$}
\psfrag{-4.6052}{ \ $10^{-2}$}
\psfrag{-2.3026}{ \ $10^{-1}$}
\includegraphics[scale=0.85]{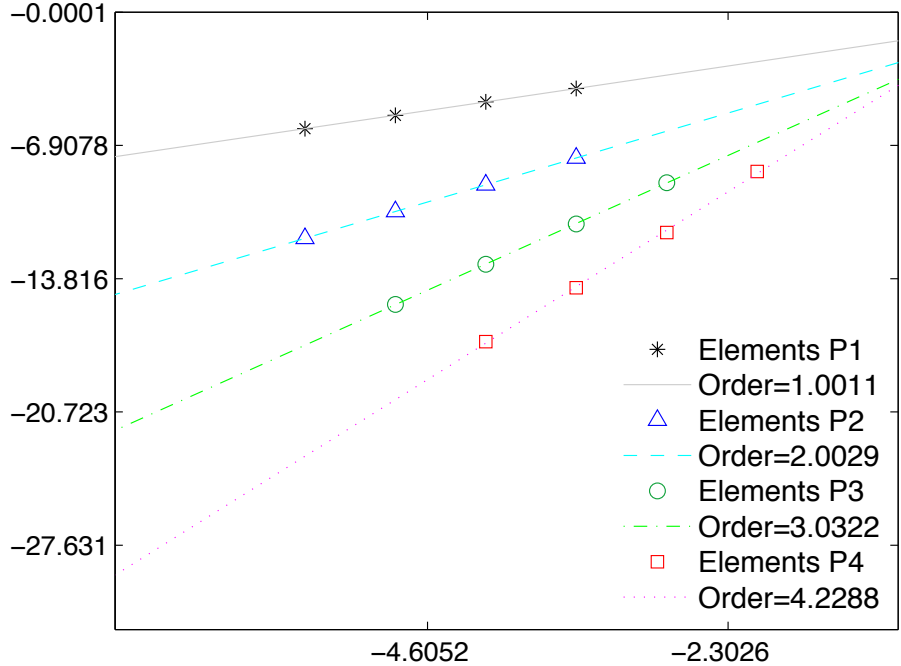}
\caption{Estimated order of convergence for $H^1(\Omega_{0})$-norm for the finite element method $P_{k}$, $k=1, 2, 3, 4$.}
\label{ordersh1}
\end{figure}


\esp

\bibliographystyle{plain}
\bibliography{biblio}

\end{document}